\documentclass[11pt,thmsa]{article}
\usepackage{amsmath}
\usepackage{amsfonts}
\usepackage{amssymb}
\usepackage{graphicx}
\newtheorem{theorem}{Theorem}[section]

\newtheorem{corollary}[theorem]{Corollary}

\newtheorem{definition}[theorem]{Definition}

\newtheorem{lemma}[theorem]{Lemma}

\newtheorem{proposition}[theorem]{Proposition}

\newenvironment{proof}[1][Proof]{\noindent\textbf{#1.} }{\ \rule{0.5em}{0.5em}}

\begin{document}
\title{Isoparameteric hypersurfaces in a Randers sphere of constant flag curvature}
\author{Ming Xu
\\
College of Mathematics\\
Tianjin Normal University\\
Tianjin 300387, P.R. China\\
Email: mgmgmgxu@163.com\\
 \\}
\date{}
\maketitle{}

\abstract{In this paper, I study the isoparametric hypersurfaces in a Randers sphere $(S^n,F)$ of constant flag curvature, with the navigation datum $(h,W)$. I prove that an isoparametric hypersurface $M$ for the standard round sphere $(S^n,h)$ which is tangent to $W$ remains isoparametric for $(S^n,F)$ after the navigation process. This observation provides a special class of isoparametric hypersurfaces in $(S^n,F)$, which can be equivalently  described as the regular level sets of isoparametric functions $f$ satisfying $-f$ is transnormal. I provide a classification for these special isoparametric hypersurfaces $M$, together with their ambient metric $F$ on $S^n$, except the case that $M$ is of the OT-FKM type with the multiplicities $(m_1,m_2)=(8,7)$. I also give a complete classificatoin for all homogeneous hypersurfaces in $(S^n,F)$. They all belong to these special isoparametric hypersurfaces. Because of the extra $W$, the number of distinct principal curvature can only be 1,2 or 4, i.e.
there are less homogeneous hypersurfaces for $(S^n,F)$ than those for $(S^n,h)$.


}

\section{Introduction}

In Riemannian geometry, the study on {\it isoparametric
hypersurfaces} has a long history. It is defined
as a regular level set for an
{\it isoparametric function} $f$ on a Riemannian manifold $(N,ds^2)$, i.e.
\begin{equation}\label{isoparametric-function-defining-condition}
|\nabla f|^2=a(f),\mbox{ and }\Delta f=b(f),
\end{equation}
in which $a(\cdot)$ is smooth and
$b(\cdot)$ is continuous.

The classification of isoparametric hypersurfaces in space forms is a classical geometric problem with a history of almost
one hundred years. Those in Euclidean and hyperbolic spaces
were classified in 1930's \cite{Ca1} \cite{Se} \cite{So}.
But for the most difficult case, those in a unit sphere, the
classification work occupied a long list of works \cite{Ca2} \cite{Ca3} \cite{CCJ} \cite{Ch2011} \cite{Ch2013}
\cite{DN1985} \cite{FKM1981} \cite{Mi2013} \cite{Mu1980} \cite{Mu1981} \cite{N1} \cite{OT1975} \cite{OT1976} \cite{TT1972}, and was recently completely solved \cite{Ch2016}. Isoparametric functions and isoparametric
hypersurfaces on other Riemannian manifolds, especially
the exotic spheres, were studied by Z. Tang and his students
\cite{QT2014}.

On the other hand, isoparametric function and isoparametric hypersurface in Finsler geometry have not been studied until recently Q. He, S. Yin
and Y. Shen purposed their definition for a Finsler
space \cite{HYS2015}, satisfying
similar conditions as in (\ref{isoparametric-function-defining-condition}). But now the gradient and Laplacian are only smoothly defined on the open set where $df\neq 0$. Generally speaking,
they are nonlinear and hard for
calculation,
because of the changing base vectors. Most of their good
properties in Riemannian geometry can not be easily generalized.

Studying and classifying isoparametric hypersurfaces in Finsler space forms, i.e. complete simply connected Finsler spaces with constant flag curvature, are interesting problems naturally generalized from Riemannian geometry. But in Finsler geometry,
the metrics for space forms can be very complicated, like the examples R. Bryant constructed on spheres \cite{Br1996}. We know very few about them except some special cases.

In \cite{HYS2015} and \cite{HYS2016},
the authors considered two special cases of Finsler space forms for the ambient space, i.e. Minkowski space (with zero flag curvature), and Funk spaces (with negative constant flag curvature). They classified the isoparametric hypersurfaces in them. However, for ambient space with positive constant flag curvature, their progress is
relatively slow.

In this paper, I will consider the (non-Riemannian) {\it Randers sphere of constant flag curvature} for the ambient space, and study a special class of isoparametric hypersurfaces in them, including all the homogeneous ones. Until now, they are the only known
examples in Finsler space forms of positive flag curvature. We
guess they are the only ones, at least for Randers spheres of constant flag curvature. But this general classification problem seems very hard.

To introduce the special isoparametric hypersurfaces studied
in this paper. We need to use the {\it navigation process} and the
celebrated works of D. Bao, C. Robles and Z. Shen on Randers spheres of constant flag curvature \cite{BRS2004}. Briefly speaking, a Randers sphere $(S^n,F)$ of constant flag curvature $\kappa$ is produced by the
navigation process with the datum $(h,W)$, such that $(S^n,h)$ is a standard Riemannian sphere with the same constant curvature $\kappa$ (thus we must have $\kappa>0$),
and $W$ is a Killing vector field for $(S^n,h)$.

The isoparametric hypersurface I have studied in this paper
are those tangent to the Killing vector field $W$ in the navigation datum $(h,W)$ for the ambient Randers sphere $(S^n,F)$ of constant flag curvature. We see a correspondence between these special isoparametric hypersurfaces of $(S^n,F)$, and those
of $(S^n,h)$ which are tangent to $W$. We summarize it as the
following theorem.

\begin{theorem}\label{main-thm}
Let $(S^n,F)$ be a Randers sphere of constant flag curvature 1,
corresponding to the navigation datum $(h,W)$. Then
for any isoparametric hypersurface $M$ of the unit sphere $S^n(1)=(S^n,h)$, we can find a Killing vector field $W$ tangent to $M$, such that $M$ is isoparametric for the Randers sphere $(S^n,F)$ defined by the navigation datum $(h,W)$.
Conversely, any isoparametric hypersurface $M$ of $(S^n,F)$ which is tangent to $W$ is isoparametric for the unit sphere $S^n(1)$.
\end{theorem}

Theorem \ref{main-thm} is a summarization of Theorem \ref{main-thm-1}, Theorem \ref{main-thm-2} and Theorem \ref{main-thm-3}. In the later three theorems, I have proved
something more.

In Theorem \ref{main-thm-1}, I have given
equivalent descriptions for those special isoparametric hypersurfaces, i.e. they are the regular level sets of isoparametric functions $f$ of a Randers sphere of constant
flag curvature, such that $-f$ is also transnormal, or isoparametric.

In Theorem \ref{main-thm-2}, I point out that any connected homogeneous
hypersurface $M$ of a Randers sphere $(S^n,F)$ of constant flag curvature, with respect to the connected isometry subgroup $K=I_o(S^n,M,F)$
of $(S^n,F)$ preserving $M$, belongs to the special isoparametric hypersurfaces in Theorem \ref{main-thm} or Theorem \ref{main-thm-1}, i.e. $M$ is tangent to the Killing vector field $W$ in the navigation datum for $F$. I provide a complete classification for these homogeneous hypersurfaces. Notice that, because of the extra $W$ in the metric datum, there are less than the ones in the Riemannian context. Especially, those with the number of distinct principal curvatures $g=3$ or $6$ do not appear. Using the theory for CK-vector fields \cite{WPX2017} \cite{XW2016}, I determine all the choices of $W$ for $M$ to be $K$-homogeneous.

In Theorem \ref{main-thm-3}, I have calculated the group $G=I_o(S^{2l-1},M,h)$ and its Lie algebra for each isoparametric
hypersurface of OT-FKM type in the unit sphere $S^{2l-1}(1)=(S^{2l-1},h)$ with multiplicities $m_1=m\leq l-m-1=m_2$.
Theorem \ref{main-thm-2} and Theorem \ref{main-thm-3} together determine all the possible choices of $W$ in the navigation datum $(h,W)$ and classify the corresponding ambient metrics for the special isoparametric hypersurfaces in Theorem \ref{main-thm}, except the case when $(m_1,m_2)=(8,7)$ (see the remarks after Theorem \ref{main-thm-3}).

This paper is organized as following. In Section 2, I briefly review the fundamental knowledge of Finsler geometry which is
needed in later discussion. In Section 3, I introduce the isoparametric function and isoparametric hypersurface in Riemannian geometry and Finsler geometry. In Section 4, I study
 the special isoparametric hypersurfaces in Randers spheres of constant flag curvature, which are tangent to the vector field in the navigation datum. In Section 5 and 6, I continue to explore them from the view points of homogeneity and Clifford system.

{\bf Acknowledgement.} I would like to thank College of Mathematics and System Science, Xinjiang University, for its
hospitality during the preparation of this paper. I would also like to thank Zhongmin Shen, Qun He, Chao Qian, and Songting Yin for helpful discussions.

\section{Preliminaries}

In this section, I first briefly summarize some fundamental concepts in Finsler geometry from \cite{BCS2000} and \cite{Sh2001}

\subsection{Minkowski norm and Finsler metric}

A {\it Minkowski norm} on a real vector space $\mathbf{V}$,
$\dim \mathbf{V}=n$, is a continuous function $F:\mathbf{V}\rightarrow[0,+\infty)$ satisfying the following
conditions:
\begin{description}
\item{\rm (1)} $F$ is a positive smooth function on $\mathbf{V}\backslash 0$.
\item{\rm (2)} $F(\lambda y)=\lambda F(y)$ for any $y\in\mathbf{V}$ and $\lambda\geq 0$.
\item{\rm (3)} With respect to any linear coordinates $y=y^i e_i\in\mathbf{V}$, the Hessian matrix
    $$g_{ij}(y)=\left(\frac{\partial^2}{\partial y^i\partial y^j}F^2(y)\right)$$
    is positive definite for any $y\neq 0$.
\end{description}

We also call $g_{ij}(y)$ the {\it fundamental tensor}. The
inverse matrix for $(g_{ij}(y))$ is denoted as $(g^{ij}(y))$.
They are used in Finsler geometry to move indices up or down.
Each Hessian matrix $(g_{ij}(y))$ with $y\neq 0$ defines
an inner product $g^F_y=\langle\cdot,\cdot\rangle_y^F$ as following,
$$\langle u,v\rangle_y^F=g_{ij}(y)u^i v^j=\frac{1}{2}
\frac{\partial^2}{\partial s\partial t}F^2(y+su+tv)|_{s=t=0},$$
for any $u=u^i e_i$ and $v=v^i e_i$ in $\mathbf{V}$.
 It is obvious that the inner product $g^F_y$ is independent of the linear coordinates.

A {\it Finsler metric} on a smooth manifold $M$ is a continuous function $F:TM\rightarrow[0,+\infty)$ which is smooth on
the slit tangent bundle $TM\backslash 0$, and its restriction
to each tangent space is a Minkowski norm. We also
call $(M,F)$ a {\it Finsler manifold} or a {\it Finsler space}.

For example, Riemannian metrics are Finsler metrics, which Hessian matrices with respect
to any {\it standard local coordinates}
$x=(x^i)\in M$ and $y=y^i\partial_{x^i}\in T_xM$, $g_{ij}(x,y)$,
only depends on $x$. In this case, we usually refer to the smooth
section $$F^2=g_{ij}(x)dx^i dx^j\in\Gamma(\mathrm{Sym}^2(T^*M))$$
as the
Riemannian metric. Generally speaking, we will only consider
non-Riemannian metrics in Finsler geometry.

Randers metrics are the most simple and important Finsler metrics. They are of the form $F=\alpha+\beta$, where
$\alpha$ is a Riemannian metric and $\beta$ is a 1-form, such that $|\beta(x)|_{\alpha}<1$ at each point $x\in M$.
The $(\alpha,\beta)$-metrics are generalizations of Randers
metrics, which are of the form $F=\alpha\phi(\beta/\alpha)$
with a positive smooth function $\phi$, and similar $\alpha$
and $\beta$ as for Randers metrics.

There is a canonical way to define a Riemannian metric from
a smooth vector field $Y$ on a Finsler space $(M,F)$.
Let $\mathcal{U}\subset M$ the open subset where $Y$ is nonvanishing. Then $Y$ defines a Riemannian metric $g^F_Y$ on
$\mathcal{U}$
by all the Hessian matrices $(g_{ij}(x,Y(x)))$, i.e. for each standard local coordinates on $\mathcal{U}$, $g^F_Y=\sqrt{g_{ij}(x,Y(x))dx^i dx^j}$. We call this metric the {\it localization} of $F$
at $Y$.

\subsection{Geodesic and flag curvature}

A {\it geodesic} on a Finsler space $(M,F)$ is a nonconstant smooth curve $c(t):I\rightarrow M$ which satisfies the locally minimizing
principle for the arc length functional
$$L(c(t))=\int_I F(c(t),\dot{c}(t))dt.$$

Any geodesic can be re-parametrized such that $F(c(t),\dot{c}(t))\equiv\mathrm{const}>0$. Then it is locally defined by the equations
$$\ddot{c}^i(t)+2\mathbf{G}^i(c(t),\dot{c}(t))=0,\quad\forall i,$$
where $\mathbf{G}^i=\frac{1}{4}([F^2]_{x^ky^l}y^k-[F^2]_{x^l})$
is the coefficient of the {\it geodesic spray}.

A smooth vector field on $(M,F)$ is a {\it geodesic field}, if
it is nonvanishing everywhere, and its integration curves are
geodesics of $(M,F)$.

Now we define the {\it flag curvature}, which is ageneralization for the sectional curvature in Riemannian geometry. Let $y$ be any nonzero tangent vector in $T_xM$ and $\mathbf{P}=\mathrm{span}\{y,v\}$ a tangent
plane in $T_xM$ containing $y$. Then the flag curvature for
$(x,y,\mathbf{P})$ is defined by
$$K^F(x,y,\mathbf{P})=K^F(x,y,y\wedge v)
=\frac{\langle R^F_y v,v\rangle_y^F}{
\langle y,y\rangle_y^F\langle v,v\rangle_y^F-
[\langle y,v\rangle_y^F]^2},$$
where $R^F_y$ is the Riemann curvature, which can be locally presented as $R_y^F=R^k_i(y)\partial_{x^i}\otimes dx^k:T_xM\rightarrow T_xM$, where
$$R_k^i(y)=2\partial_{x^k}\mathbf{G}^i
-y^j\partial^2_{x^jy^k}\mathbf{G}^i+
2\mathbf{G}^j\partial^2_{y^jy^k}
\mathbf{G}^i-\partial_{y^j}
\mathbf{G}^i\partial_{y^k}\mathbf{G}^j.$$

When $F$ is Riemannian, the flag curvature $K^F(x,y,\mathbf{P})$ coincides with the sectional curvature, which only depends on the tangent plane $\mathbf{P}\in T_xM$.

\subsection{Navigation process and Randers spheres of constant flag curvature}

The navigation process is an important method for constructing new metrics from an old ones.

Let $F$ be a Finsler metric and $W$ a vector field on $M$ such that $F(W(x))<1$ for each $x\in M$. We denote $\tilde{y}=y+F(x,y)W(x)$ for any $y\in T_xM$. Then the equality
$\tilde{F}(\tilde{y})=F(y)$ defines another Finsler metric
which indicatrix $\mathcal{I}^{\tilde{F}}_{x}=\{\tilde{y}\in T_xM|\tilde{F}(x,y)=1\}$ is a parallel shifting of the indicatrix
$\mathcal{I}^F_{x}$ by the vector $W(x)$.

We call $(F,W)$ the
navigation datum defining $\tilde{F}$. When $F$
is Riemannian, $\tilde{F}$ is a Randers metric, which can be
presented as
$$\tilde{F}(\cdot)=\frac{1}{\lambda}(\sqrt{\lambda F(\cdot)^2+(\langle W,\cdot\rangle^F)^2}-\langle W,\cdot\rangle^F),$$
where $\lambda=1-F(W)^2$. All Randers
metrics can be produced in this way, and there is
a one-to-one correspondence between Randers metrics
and navigation data $(h,W)$ in which $h$ is a Riemannian metric.

The navigation process is crucial for studying Randers metrics of constant flag curvatures. The following theorem summarized from \cite{BRS2004} provides the foundation.

\begin{theorem} \label{classification-of-constant-curvature-Randers-space}
A Randers metric $F$ has the constant flag
curvature $\kappa$ iff its navigation datum $(h,W)$
satisfies the following conditions:
\begin{description}
\item{\rm (1)} The metric $h$ has constant curvature $k+\frac14\mu^2$ for some constant $\mu$.
\item{\rm (2)} The vector field $W$ is $\mu$-homothetic for
$h$, i.e. $\mathcal{L}_Wh=2\mu h$, where $\mathcal{L}$ is the
Lie derivative.
\end{description}
Furthermore, for a Randers sphere of constant flag curvature
with the navigation datum $(h,W)$, $W$ can only be a Killing vector field, i.e. $0$-homothetic for the metric $h$.
\end{theorem}

Randers spheres of constant flag curvature have many good geometric properties. For example, its S-curvature vanishes \cite{BR2003} \cite{Xi2005}.

\section{Isoparametric function and isoparametric hypersurface}

In this section, I will recall the definitions of isoparametric function and isoparametric hypersurface in Riemannian geometry
and very briefly discuss their classification when the ambient space is a unit sphere. Then I will introduce their generalization in Finsler geometry.

\subsection{Definitions in Riemannian geometry}

The key feature of an {\it isoparametric function} $f$ in Riemannian
geometry is that $|\nabla f|$ and $\Delta f$ only depend on values of $f$. Any regular level set of $f$ (i.e. pre-image of regular values of $f$) is called an {\it isoparametric hypersurface}.

In practice, for some technical reasons (see \cite{Wa1987} for example), we use the following definitions.
A nonconstant smooth function $f$ on a Riemannian manifold $N$ is called
{\it transnormal} if $|\nabla f|^2=a(f)$ for some smooth function $a(\cdot)$. Furthermore a transnormal function
$f$ is called {\it isoparametric} if $\Delta f=b(f)$ for some
continuous function $b(\cdot)$.

For example, if $N$ is simply connected and it admits the cohomogeneity one isometric action of a compact connected Lie group $G$, then we can find a suitable $G$-invariant isoparametric function $f$, such that each principal $G$-orbit is a {\it homogeneous isoparametric hypersurface}.

The classification for isoparametric hypersurfaces in unit spheres is one of the most celebrated geometric problems. It has been studied for eighty
years, and completely solved recently \cite{CCJ} \cite{Ch2011} \cite{Ch2013} \cite{Ch2016}
\cite{Mi2013}. Briefly speaking, it is either homogeneous or of the OT-FKM type. We will see more detailed
descriptions for these two cases in Section 5 and Sectoin 6
respectively.

Any connected imbedded orientable hypersurface in a unit sphere with constant mean curvature is an isoparametric hypersurface. Its principal curvatures, counting multiplicities, are constant functions too, which can be expressed as $\kappa_i=\cot(\theta+\frac{i-1}{g})$, with multiplicities $m_i$, and $m_i=m_{i+2}$ where the subindices are mod $g$.
The only possible $g$'s are 1, 2, 3, 4 and 6.

\subsection{Definitions in Finsler geometry}

Now we define a
smooth function $f$ on a Finsler space $(N,F)$ to be isoparametric, which needs proper interpretations for $\nabla f$ and $\Delta f$.

We assume the smooth function $f$ is not constant, i.e. the open subset
$\mathcal{U}=\{x\in N|df(x)\neq0\}$ of $M$ is nonempty. Each level set of $f$ in $\mathcal{U}$ is then a smooth hypersurface. Assume $x\in M=\mathcal{U}\cap f^{-1}(c)$, then
there exists a unique vector $\nabla f(x)$ pointing to the increasing direction of $f$ and satisfying
$$\langle\nabla f(x),v\rangle^F_{\nabla f(x)}=df(v)|_{x}.$$
Then $\nabla f$ defines a smooth vector field on $\mathcal{U}$
which can be continuously extended to $M\backslash \mathcal{U}$
where it equals 0. We call $\nabla f$
the {\it nonlinear gradient} of $f$. It can also be interpreted as the gradient vector field of $f|_{\mathcal{U}}$ with respect to the localization metric $g^F_{\nabla f}$. Notice that generally $-\nabla f\neq \nabla (-f)$.

Similarly, we also use the metric $g^F_{\nabla f}$ to define the {\it nonlinear Laplacian} $\Delta f$ of $f$, where $\Delta$ is the Laplacian with respect to $g^F_{\nabla f}$. Now we are ready to give the following definitions.

\begin{definition}\label{define-isoparametric-Finsler}
A nonconstant smooth function $f$ on $M$ is called transnormal if $F(\nabla f)$
only depends on values of $f$ when they are restricted to $\mathcal{U}=\{x\in M|df(x)\neq 0\}$. Further more a transnormal function $f$ is called isoparametric
if $\Delta f$ satisfies the same property on $\mathcal{U}$.
Each level set of $f$ in $\mathcal{U}$ is called an
isoparametric hypersurface.
\end{definition}

The isoparametric condition we propose here is simpler
than that in \cite{HYS2015}. Two
definitions apply the same idea that using the nonlinear gradient $\nabla f$ as the base vector. But their definition for $\Delta f$ involves
an arbitrary volume form on $M$, which results an extra
summation term $S(\nabla f)$, where $S(\cdot)$ is the S-curvature with respect to the chosen volume form \cite{Sh2001}. If we choose the BH-volume and discuss the case that
the ambient manifold is a Randers sphere of constant flag curvature, there is no difference between the two definitions because the S-curvature vanishes.

The transnormal condition can be equivalently presented as
\begin{equation}\label{Finsler-transnormal-condition}
F(\nabla f(x))=a(f(x)),\quad x\in \mathcal{U},
\end{equation}
where $a(\cdot)$ is a smooth function on $f(\mathcal{U})$.
By Sard theorem $f(\mathcal{U})$ is a dense open set in $f(N)$, and by the transnormal condition $\mathcal{U}=f^{-1}(f(\mathcal{U}))$.
But $f$ may have many critical values, which can not be erased
as in Riemannian geometry by adjusting $f$, because $-f$ may not be transnormal.

For example, in an Minknowski space $(\mathbb{R}^n,F)$ with $n>1$,  $S^{n-1}_F(x_0,r)=\{x\in\mathbb{R}^n|d_F(x_0,x)=r\}$ with $r>0$ is an
isoparametric hypersurface, corresponding to the isoparametric
function $f(x)=d_F(x_0,x)$ \cite{HYS2015}. If $-f$ is transnormal, by (1) of Lemma \ref{easy-lemma-for-double-transnormal} below, the dual
norm $F^*$ of $F$ must be reversible, i.e. $F^*(y)=F^*(-y)$, which implies $F$ is also
reversible.

This observation suggests it be an interesting problem to study and classify the special isoparametric hypersurface which isoparametric function $f$ satisfy $-f$ is also transnormal (or more strongly, $-f$ is also isoparametric).

\section{Special isoparametric hypersurfaces in Randers spheres of constant flag curvature}

In this section, we study the special isoparametric hypersurfaces proposed at the end of last section,
when the ambient manifold is a Randers sphere $(S^n,F)$ of constant flag
curvature. We will see they are equivalently characterized by
the condition that they are tangent to the Killing vector field
$W$ in the navigation datum $(h,W)$ for $F$.

Firstly, we consider a nonconstant smooth function $f$ on a Finsler space $(N,F)$
such that $f$ and $-f$ are both transnormal.

There exist
positive smooth function $a_1(\cdot)$ and $a_2(\cdot)$ such that $F(\nabla f)=a_1(f)$ and $F(\nabla (-f))=a_2(f)$ on
$\mathcal{U}=\{x\in N|df(x)\neq 0\}$. Denote
$$\mathbf{n}_1=\frac{\nabla f}{a_1(f)}\quad\mbox{and}\quad
\mathbf{n}_2=\frac{\nabla(-f)}{a_2(f)}$$
the two unit normal fields along all level sets of $f$ in
$\mathcal{U}$. Then we have the following easy lemma.
\begin{lemma}\label{easy-lemma-for-double-transnormal}
(1)
The function
$f$ defines a Finsler submersion from $(\mathcal{U},F)$ to the one-dimensional
manifold $f(\mathcal{U})$, such that the induced Finsler metric $F'$ on
$f(\mathcal{U})$ satisfy
$F'(f_*\mathbf{n}_1)=F'(f_*\mathbf{n}_2)=1$.

(2)
The vector fields $\mathbf{n}_1$ and $\mathbf{n}_2$ are
geodesic fields on $(\mathcal{U},F)$ (i.e. their integration curves are geodesics of $(\mathcal{U},F)$).
\end{lemma}

See \cite{PD2001} for the theory of Finsler submersion. Here
we only use its definition, and the fact that the horizonal lift of a geodesic (field) is a geodesic (field). The proof
of the lemma is very easy, so we omit it.

Notice if we only have the transnormal property for $f$,
the function $f$ can still be treated as a "submersion for
just one side", so that only $\mathbf{n}_1$ is a horizonal lift of a geodesic field on $f(\mathcal{U})$. So $\mathbf{n}_1$ is
geodesic field for $(\mathcal{U},F)$, but $\mathbf{n}_2$ is not in general.

Secondly, we further assume the ambient manifold $(N,F)$ is
a Randers space, which corresponds to the navigation datum
$(h,W)$.

Let $c$ be the positive smooth function on $\mathcal{U}$
defined by
$f_*\mathbf{n}_1=-c f_*\mathbf{n}_2$. Then by Lemma
\ref{easy-lemma-for-double-transnormal}, $c$ only depends on values of $f$. Direct calculation shows $W=\frac12(\mathbf{n}_1+\mathbf{n}_2)$, and thus
$f_* W=\frac12(f_*\mathbf{n}_1+f_*\mathbf{n}_2)$ is a well defined vector field on $f(\mathcal{U})$.

Thirdly, we further require the ambient space is a Randers sphere $(S^n,F)$ of constant flag curvature, i.e. in its navigation datum $(h,W)$, the Riemannian
metric $h$ has a positive constant curvature and $W$ is a Killing vector field for $(S^n,h)$. Then we claim

\begin{lemma}\label{tangent-lemma}
If $f$ is a nonconstant smooth function on a Randers sphere of constant flag curvature, such that both $f$ and $-f$ are transnormal, then $f$ is preserved by the flow generated by $W$.
\end{lemma}

\begin{proof}
We first observe that $f_*W\equiv 0$ in $\mathcal{U}$, i.e. the flow generated by $W$ preserves each regular value of $f$. Assume
conversely $f_*W\neq 0$ at the pre-image for the regular value
$c$ of $f$, then a diffeomorphism generated by $W$ maps
$N^c=\{x\in N|f(x)\leq c\}$ to some other $N^{c'}$ with $c'\neq c$. It is a contradiction because $W$ is Killing vector field of $(N,h)$, i.e. it generates isometries which preserves the volume defined by $h$.

Because $f(\mathcal{U})$ is a dense open set in $f(S^n)$, by
continuity, critical values of $f$ are also preserved by the
flow generated by $W$, which ends the proof for the lemma.
\end{proof}

Fourthly, we consider the isoparametric condition and prove the following lemma.

\begin{lemma}
If $f$ is an isoparametric function on a Randers sphere $(S^n,F)$ of constant flag curvature, with the navigation datum $(h,W)$, such that both $f$ and $-f$ are transnormal, then
$f$ is also isoparametric for the Riemannian metric $h$.
\end{lemma}

When restricted to $\mathcal{U}$, Lemma \ref{tangent-lemma} implies that $W$ is tangent to each level set of $f$. Denote
$\mathbf{n}=\nabla^h f/|\nabla^h f|$ the unit vector field in $\mathcal{U}$, where $\nabla^h$ is the gradient operator with respect to $h$. Then $\mathbf{n}_1=\mathbf{n}+W$, $\mathbf{n}_2=-\mathbf{n}+W$ and $[\mathbf{n},W]=0$.

At any point in $\mathcal{U}$, we can find local coordinates
$x=(x^i)$, $1\leq i\leq n$, such that
$\partial_{x^1}=\mathbf{n}$, $\partial_{x^2}=W$ and the level sets of $f$ are given by the equation $x^1\equiv\mathrm{const}$. Then the metric $h$ can be locally presented
as $$h=\sqrt{(dx^1)^2+\sum_{i,j>1}h_{ij}dx^i dx^j},$$
where the functions $h_{ij}$ do not depend on $x^2$, and $0\leq h_{22}<1$ (because $W$
is a Killing vector field of $h$ which length at each point is strictly less than 1).

Using above local coordinates, it is obvious to see $f$ is
transnormal with respect to $h$.

To describe the nonlinear Laplacian $\Delta f$, where $\Delta$
is the Laplacian operator with respect to the localization  $g^F_{\nabla f}=g^F_{\mathbf{n}_1}$, we define another local coordinate
system at $x$ as following,
$$\tilde{x}^2=x^2-x^1,\mbox{ and }\tilde{x}^i=x^i\mbox{ when }i\neq 2.$$
Then $\partial_{\tilde{x}^1}=\mathbf{n}_1$ and the level sets of $f$ are given by $\tilde{x}^1\equiv\mathrm{const}$. So we can similarly present $$g^F_{\mathbf{n}_1}=\sqrt{(d\tilde{x}^1)^2+
\sum_{i,j>1}g_{ij}d\tilde{x}^i d\tilde{x}^j},$$
where $g_{ij}=g^F_{ij}(x,\mathbf{n}_1)$. To see the relation between $g_{ij}$ and $h_{ij}$, we need the following lemma.

\begin{lemma} \label{lemma-on-navigation}
Let $\tilde{F}$ be the Minkowski norm on $\mathbf{V}$ defined by the navigation datum $(F,v)$, i.e.
$\tilde{F}(\tilde{y})=F(y)$ where $\tilde{y}=y+F(y)v$ for
each $y\in \mathbf{V}$. Then for any vector $y$ and $u$ in $\mathbf{V}$ satisfying $y\neq 0$ and
$\langle u,y\rangle_y^F=0$, we have
\begin{equation}\label{0000}
\langle u,u\rangle_y^F=
\frac{\langle u,u\rangle_{\tilde{y}}^{\tilde{F}}}{1-
\langle\tilde{y},v\rangle_{\tilde{y}}^{\tilde{F}}}.
\end{equation}
In particular, $\langle u,u\rangle_{\tilde{y}}^{\tilde{F}}
={\langle u,u\rangle_y^F}$ for all $u$ in the $g^F_y$-complement of $y$ if $\langle v,y\rangle_y^F=0$.
\end{lemma}
\begin{proof}
For any $t$, we have
$$\tilde{F}^2(y+tu+F(y+tu)v)=F^2(y+tu).$$
Differentiate it twice for the $t$-variable, and take $t=0$, then we get
$$\langle u,u\rangle_{\tilde{y}}^{\tilde{F}}+
\langle \tilde{y},\langle u,u\rangle_y^F \cdot v\rangle_{\tilde{y}}^{\tilde{F}}=\langle u,u\rangle_y^F,$$
and
$$\langle u,u\rangle_y^F=
\frac{\langle u,u\rangle_{\tilde{y}}^{\tilde{F}}}{1-
\langle\tilde{y},v\rangle_{\tilde{y}}^{\tilde{F}}}.$$

Because the indicatrix $\mathcal{I}^{\tilde{F}}$ is a parallel
shifting of the indicatrix $\mathcal{I}^F$, we have
$\langle \tilde{y},v\rangle_{\tilde{y}}^{\tilde{F}}=0$ when
 $\langle y,v\rangle_y^F=0$. So we have
 $\langle u,u\rangle_{\tilde{y}}^{\tilde{F}}
={\langle u,u\rangle_y^F}$ in this case.
\end{proof}

Notice that in Lemma \ref{lemma-on-navigation}, the statuses of $F$ and $\tilde{F}$ are symmetric, i.e. $({\tilde{F}},-v)$ is also the navigation datum for $F$. We have
$\langle u,\tilde{y}\rangle_{\tilde{y}}^{\tilde{F}}=0$ when
$\langle u,y\rangle_y^F=0$. So (\ref{0000}) can also be given
as
$$\langle u,u\rangle_{\tilde{F}}^{\tilde{y}}
=\frac{\langle u,u\rangle_y^F}{1+\langle y,v\rangle_y^F}.$$

By Lemma \ref{lemma-on-navigation}, we have
\begin{eqnarray*}
g_{ij}(\tilde{x}^1,\tilde{x}^2,\ldots,\tilde{x}^n)&=&
h_{ij}(x^1,x^2,\ldots,x^n)\\
&=&
h_{ij}(\tilde{x}^1,\tilde{x}^2+\tilde{x}^1,
\tilde{x}^3,\ldots,\tilde{x}^n)\\
&=&h_{ij}(\tilde{x}^1,\tilde{x}^2,\tilde{x}^3,\ldots,\tilde{x}^n),
\end{eqnarray*}
where $i,j>1$.

To summarize, we see that $g^F_{\mathbf{n}_1}$ has the same local presenting as $h$, except that all $x^i$'s are changed to
$\tilde{x}^i$'s respectively, and by similar argument, we can prove so does $g^F_{\mathbf{n}_2}$. All three metrics, $h$, $g^F_{\mathbf{n}_1}$ and $g^F_{\mathbf{n}_2}$ are Riemannian metrics of the same positive
constant curvature. By the isoparametric condition, the level sets of $f$ in $\mathcal{U}$
has constant principal curvatures with respect to $g^F_{\mathbf{n}_1}$, thus the statement is valid with $g^F_{\mathbf{n}_1}$ changed to the other two, i.e. $f$
is isoparametric for $(S^n,h)$, and $-f$ is isoparametric for
$(S^n,F)$.

Finally, we consider an isoparametric function $f$ for $(S^n,h)$
which is preserved by the flow generated by $W$, where $(h,W)$
is the navigation datum for a Randers sphere $(S^n,F)$ of constant flag
curvature. With very minor changes, above argument also proves
both $f$ and $-f$ are isoparametric for $(S^n,F)$.

Summarizing above arguments, we have the following theorem.
\begin{theorem}\label{main-thm-1}
Let $(S^n,F)$ be a Randers sphere of constant flag curvature, with the navigation datum $(h,W)$. Then any isoparametric function
$f$ of $(S^n,h)$ which is preserved by the flow generated by $W$ is isoparametric for $(S^n,F)$.

Conversely, if $f$ is isoparametric and $-f$ is transnormal for $(S^n,F)$, then
$f$ is isoparametric for $(S^n,h)$ and preserved by the flow
generated by $W$. Furthermore, $-f$ is also parametric for
$(S^n,F)$.
\end{theorem}

Theorem \ref{main-thm-1} implies when the ambient space is an Randers sphere of constant flag curvature, the special  isoparametric hypersurfaces purposed in last section are just those in the classification list in the Riemannian context
which are tangent to the vector field $W$ in the navigation datum, or equivalently, preserved by the flow generated by $W$.
We will see in the next two sections, that each isoparametric hypersurfaces in a unit sphere permits navigation changes for the metric, i.e. Killing vector field tangent to it. Further more, all the Killing vector field $W$ tangent to an isoparametric hypersurface for the standard round sphere $(S^n,h)$ (i.e. all the Randers metrics $F$ of constant flag curvature permitting these special isoparametric hypersurfaces)
can be determined except the case of the OT-FKM type with multiplicites $(m_1,m_2)=(8,7)$.

\section{Homogeneous isoparametric hypersurface}

Assume $(S^n,F)$ is a Randers sphere of constant flag curvature, and $(h,W)$ is its navigation datum.
Without loss of generalities, we may assume
$(S^n,h)$ is the unit sphere $S^n(1)$ and $W$ is a nonzero Killing vector
field for it. The connected isometry group $I_o(S^n,F)$ of $(S^n,F)$ is the proper subgroup
of $I_o(S^n,h)=SO(n+1)$ which preserves $W$, or equivalently
commutes with $W$ when $W$ is viewed as a matrix in $so(n+1)$.

Assume
$M$ is a connected homogeneous isoparametric hypersurface
of $S^n(1)=(S^n,h)$. We denote $G=I_o(S^n,M,h)$ and $K=I_o(S^n,M,F)\subset G$ the maximal connected subgroup of
$SO(n+1)$ and $I(S^n,F)$ respectively, preserving $M$.
The homogeneity here, in the Riemannian context, means that $G$ acts transitively on $M$.
All $G$'s are classified by the following table, which coincides with
the isotropy actions of compact rank two symmetric spaces \cite{HL}.

\begin{table}
  \centering
  \begin{tabular}{|l|l|l|l|l|}
  \hline
  Case & $G=I_o(S^n,M,h)$ & $\dim M$  & $g$ & $(m_1,m_2)$ \\ \hline
  1 & $SO(n)$ & $n$ &  $1$ & $(n,n)$ \\ \hline
  2 & $SO(p)\times SO(n+1-p)$ & $n$ & $2$ & $(p,n-p)$ \\ \hline
  3 & $SO(3)$ & $3$ & $3$ & $(1,1)$ \\ \hline
  4 & $SU(3)$ & $6$ & $3$ & $(2,2)$ \\ \hline
  5 & $Sp(3)$ & $12$ & $3$ & $(4,4)$ \\ \hline
  6 & $F_4$ & $24$ & $3$ & $(8,8)$ \\ \hline
  7 & $SO(5)$ & $8$ & $4$ & $(2,2)$ \\ \hline
  8 & $U(5)$ & $18$ & $4$ & $(4,5)$ \\ \hline
  9 & $SO(m)\times SO(2)$, $m\geq 3$ & $2m-2$ & $4$ & $(1,m-2)$ \\ \hline
  10& $S(U(m)\times U(2))$, $m\geq 2$ & $4m-2$ & $4$ & $(2,2m-3)$ \\ \hline
  11& $Sp(m)\times Sp(2)$, $m\geq 2$ & $8m-2$ & $4$ & $(4,4m-5)$ \\ \hline
  12& $(Spin(10)\times SO(2))/\mathbb{Z}_4$ & $30$ & $4$ & $(6,9)$ \\ \hline
  13& $SO(4)$ & $6$ & $6$ & $(1,1)$ \\ \hline
  14& $G_2$ & $12$ & $6$ & $(2,2)$ \\
  \hline
\end{tabular}
\end{table}
\bigskip

The metric $h$ defines a norm $||\cdot||_h$
on $\mathfrak{g}=\mathrm{Lie}(G)$ such that
$$||W||_h=\max_{x\in S^n(1)}|W(x)|_h$$
for each $W\in\mathfrak{g}=\mathrm{Lie}(G)$ viewed as Killing vector field of
$S^n(1)$. Viewing $W$ as a matrix in $so(n+1)$ instead, then we have
\begin{eqnarray*}
||W||_h&=&\max\{\langle Wx,Wx\rangle\mbox{ for all }x\in S^n(1)\subset\mathbb{R}^{n+1}\}\\
&=&\max\{c|\pm c\sqrt{-1}\mbox{ are eigenvalues of }W\},
\end{eqnarray*}
where $\langle\cdot,\cdot\rangle$ is the standard Euclidean inner product.

Theorem \ref{main-thm-1} indicates that, if we
choose
the Killing vector field $W$ from the open $||\cdot||_h$-unit ball in $\mathfrak{g}$, i.e. $||W||_h<1$, $M$ is tangent to $W$ and remains isoparametric after the navigation.

The subtle issue here is that $M$ may not be $K$-homogeneous, i.e. a homogeneous hypersurface of $(S^n,F)$.
The following lemma indicates any $K$-homogeneous hypersurface in $(S^n,F)$ are tangent to $W$, i.e. an isoparametric hypersurface indicated in Theorem \ref{main-thm-1}.
\begin{lemma}\label{homogeneous-lemma}
If $M$ is $K$-homogeneous, then it is tangent to the vector field $W$ in the navigation datum.
\end{lemma}
\begin{proof}
When $W$ is viewed as a matrix in $\mathfrak{g}\subset so(n+1)$, it commutes
with $\mathfrak{k}=\mathrm{Lie}(K)$.
We only need to prove $W\in\mathfrak{k}$. Assume conversely
that $W\notin\mathfrak{k}$, then the closure
$K'$ in $G$ for the group generated by the Lie subalgebra $\mathbb{R}W\oplus\mathfrak{k}$ (which is a closed subgroup of $G$) acts transitively on $S^n$. The semi-simple part of $\mathfrak{k}'=\mathrm{Lie}(K')$ coincides with that of $\mathfrak{k}$. According to the classification of effective
transitive group actions on spheres \cite{Bo1940}
\cite{MS1943}, the semi-simple part of $\mathfrak{k}$ must be  $su(m)$ when $n+1=2m$ is an even number,
or $sp(m')$ when $n+1=4m'$ can be divided by 4. In either
case, $K$ acts transitively on $S^n$, which can not preserve
$M$. This is a contradiction which ends the proof of the lemma.
\end{proof}

The remaining task of this section is to classify all
$K$-homogeneous hypersurfaces, i.e. all homogeneous isoparametric hypersurfaces of $(S^n,F)$.

Assume $M\subset S^n$ is a $K$-homogeneous hypersurface.
By a suitable orthogonal-conjugation, we can identify the
Killing vector field $W$ with the matrix
$$W=\mathrm{diag}(0_{n_0},\lambda_1 J_{2n_1},\ldots,\lambda_k J_{2n_k}),$$
where $0<\lambda_1<\lambda_2<\cdots<\lambda_k<1$, $0_{n_0}$ is the $n_0\times n_0$-zero matrix, and $J_{2n_i}$ is the skew symmetric $2n_i\times 2n_i$-matrix
 $J_{2n_i}=\left(
                \begin{array}{cc}
                  0 & I_{n_i} \\
                  -I_{n_i} & 0 \\
                \end{array}
              \right)
$, $n_i>0$ when $i>0$, and $n_0$ can be 0.

When $n_0>0$ and $k=2$, or $k>2$, the action of $I_o(S^n,F)$ on $S^n$ has a cohomogeneity bigger than 1. So in this case, $M$ can not be $K$-homogeneous.

When $n_0=0$ and $k=2$, or $n_0>1$ and $k=1$, or $n_0=1$ and $k=1$, $I_o(S^n,F)=U(m_1)\times U(m_2)$ with $2m_1+2m_2=n+1$, or $SO(n_0)\times U(m_1)$ with $n_0+2m_1=n+1$, or $U(m)$ with $2m+1=n$ respectively.
In either case, $W$ is in the center of $\mathrm{Lie}(I_o(S^n,F))$, and $K=I_o(S^n,F)$. The action of $I_o(S^n,F)$ on $S^n$ is of cohomogeneity one, so the $K$-homogeneous $M$ must be an orbit of it, which is an isoparametric hypersurface in $S^n(1)$ with 1 or 2 distinct principal curvatures.

When $n_0=0$ and $k=1$,  $I_o(S^n,F)=U(m)$ with $2m=n+1$ acts transitively on $S^n(1)$. In this case, $W$ is a {\it Killing vector field of constant length} (or {\it CK-vector field} in short) on $S^n(1)$. Notice $K$ is the connected centralizer of $W$ in $G$. The assumption that $M$ is $K$-homogeneous implies that if we present $M$ as the $G$-homogeneous space $M=G/H$ with
$\mathfrak{h}=\mathrm{Lie}(H)$, then
\begin{equation}\label{0001}
\mathfrak{g}=\mathfrak{h}+\mathfrak{k}.
\end{equation}
This equality has been studied by A. L. Oniscik \cite{On1962}. A more geometric way to interprete it, which will be applied in later discussions, is that the restriction of $W$ to $M$ defines a nonzero Killing vector field of constant length (or CK-vector field in short) on $G/H$, with respect to all $G$-homogeneous metrics. In particular, we can choose the Riemannian $G$-normal homogeneous metric for $G/H$.

This observation is valid for all hypersurfaces and focal manifolds in the isoparametric foliation associated to $M$.
Assume $M$ is a regular
level set for the
isoparametric function $f:S^n(1)\rightarrow[-1,1]$ such that $f^{-1}(\pm 1)$ are the two focal submanifold. Then for each $t\in[-1,1]$, $K$ preserves and acts transitively on the level set $f^{-1}(t)$. If we present $f^{-1}(t)=G/H_t$ with $\mathrm{Lie}(H_t)=\mathfrak{h}_t$, then the restriction of $W$ to each
$f^{-1}(t)$ defines a CK-vector field on $G/H_t$, with respect to all $G$-homogeneous metrics, the Riemannian $G$-normal
homogeneous ones.

We will use the following two results for CK-vector fields on Riemannian normal homogeneous spaces.

\begin{theorem}
\label{normal-homogeneity-CK-vector-field-simple-case}
Let $G$ be a compact connected simple Lie group and $H$ a closed subgroup with $0<\dim H<\dim G$. Fix a Riemannian normal
metric on $G/H$. Suppose that there is a nonzero vector $v\in\mathfrak{g}$ defining a CK-vector field on $G/H$. Then
up to a finite cover, the only possibilities of $G/H$ are
$S^{2n-1}=SO(2n)/SO(2n-1)$, $S^7=Spin(7)/G_2$ and $SU(2n)/Sp(n)$.
\end{theorem}

\begin{proposition} \label{normal-homogeneity-CK-vector-field-semi-simple-case}
Assume $G/H$ is an
irreducible Riemannian normal
homogeneous space with $G$ compact and semisimple. Denote
$$\mathfrak{g}=\mathfrak{g}_1\oplus\cdots\oplus\mathfrak{g}_r$$
the direct sum
decomposition of $\mathfrak{g}=\mathrm{Lie}(G)$, where each $\mathfrak{g}_i$ is simple. Denote $\pi_i$ the projection from
$\mathfrak{g}$ to $\mathfrak{g}_i$ according to this decomposition. Assume for each $i$, the $\pi_i$-image of
$\mathfrak{h}=\mathrm{Lie}(H)$ has a dimension strictly between 0 and $\dim \mathfrak{g}_i$, and $\xi=\xi_1+\cdots+\xi_r
\in\mathfrak{g}$ defines a CK-vector field on the normal
homogeneous space $G/H$. Then for each $i$, $\xi_i$ defines
a CK-vector field on the Riemannian normal homogeneous space
$G_i/H_i$ where $\mathrm{Lie}(G_i)=\mathfrak{g}_i$ and
$\mathrm{Lie}(H_i)=\pi_i(\mathfrak{h})$.
\end{proposition}

Theorem \ref{normal-homogeneity-CK-vector-field-simple-case}
is a restatement of Theorem 1.1 in \cite{XW2016}, and Proposition \ref{normal-homogeneity-CK-vector-field-semi-simple-case}
is part of Theorem 7.6 in \cite{WPX2017} for only the Riemannian case.

In the following, we discuss case by case for all fourteen possible choices of $G=I_o(S^n,M,h)$ in the table above.

In Case 1 with $G=SO(n)$, the focal manifolds are zero dimensional, which do not admit nonzero CK-vector fields.

In Case 2 with $G=SO(p)\times SO(n+1-p)$, the focal manifolds are
two unit spheres. So both $p$ and $n+1-p$ must be even, i.e. both
focal manifolds are odd dimensional spheres, so that they can admit
nonzero CK-vector fields. In this case $K=U(m_1)\times U(m_2)$
with $2m_1+2m_2=n+1$ and the CK-vector field $W$ is chosen from
the center of $\mathrm{Lie}(I_o(S^n,F))=u(m_1+m_2)$.

For the cases from Case 3 to Case 7, and Case 14, the group
$G$ is simple and $W\in\mathfrak{g}$ defines a nonzero CK-vector
field on the Riemannian normal homogeneous spaces $G/H_{\pm 1}$ (they are the focal submanifolds, but not endowed with the submanifold metric). It is a contradiction because they are not the ones listed in Theorem \ref{normal-homogeneity-CK-vector-field-simple-case}.

In Case 13, $G=SO(4)$ has the same dimension as $M$. So
its proper subgroup $K$ can not act transitively on $M$.

In Case 11 with $G=Sp(m)\times Sp(2)$ and $m\geq 2$, one of the focal manifold is $$(Sp(m)\times Sp(2)/(\Delta Sp(1) \times Sp(m-1)\times Sp(1)),$$
where $\Delta Sp(1)$ is a diagonal $Sp(1)$ in the product $Sp(m)\times Sp(2)$, and $Sp(m-1)$ and the other $Sp(1)$ are contained in $Sp(m)$ and $Sp(2)$ respectively. Theorem
\ref{normal-homogeneity-CK-vector-field-semi-simple-case}
indicates $W$ defines a nonzero CK-vector field on either  $\mathbb{H}\mathrm{P}^{m-1}=Sp(m)/Sp(m-1)Sp(1)$
or $\mathbb{H}\mathrm{P}^1=Sp(2)/Sp(1)Sp(1)$. But neither one admits
such a CK-vector field by
Theorem \ref{normal-homogeneity-CK-vector-field-simple-case}.

For all the remaining cases, $G$ contains a one-dimensional center. If we take $W$ from the center of $\mathfrak{g}$,
then obviously $K=G$ acts transitively on $M$, i.e. $M$ is
a homogeneous isoparametric hypersurface in $(S^n,F)$.
I will prove that $W$ can only be chosen
from the center. In the following cases, I will assume conversely
that $W$ is not in the center, and prove the claim by contradiction.

In Case 8 with $G=U(5)$, then the maximal possible dimension for $K$ is 17 (which is taken by $K=U(4)\times U(1)$), i.e. $K$ can not act
transitively on $M$ which dimension is 18.

In Case 9 with $G=SO(m)\times SO(2)$ and $m\geq 3$, we denote $M=G/H$ and
$H'$ is projection image of $H$ in the $SO(m)$-factor. Then
$K=K'\times SO(2)$, and we can get
$\mathfrak{k}'+\mathfrak{h}'=so(m)$
from (\ref{0001}), in which $\mathfrak{k}'=\mathrm{Lie}(K')$ and $\mathfrak{h}'=\mathrm{Lie}(H')$. So the $so(m)$-factor $W'$ of $W$
defines a nonzero CK-vector field on the Riemannian normal homogeneous space $SO(m)/H'$. Notice that
$2m-3\leq\dim SO(m)/H'\leq 2m-2$. When $m>4$, it is not listed in Theorem \ref{normal-homogeneity-CK-vector-field-simple-case}.
When $m=3$ or $4$, it can be directly checked as in Case 8 that $\dim K<\dim M$, i.e. $K=I_o(S^n,M,F)$
can not act transitively on $M$.

In Case 12 with $G=(Spin(10)\times SO(2))/\mathbb{Z}_4$ and $\dim M=30$, we can apply similar argument as for Case 9 with $m>4$ to get a contradiction.

In Case 10 with $G=S(U(m)\times U(2))$ and $m\geq 2$,
one of the focal manifold corresponds to the $G$-orbit in
$$\left(
    \begin{array}{cccc}
      1 & 0 & \cdots & 0 \\
      0 & 0 & \cdots & 0 \\
    \end{array}
  \right)\in\mathbb{C}_{2\times m},
$$
on which $U(m)$ acts on the columns and $U(2)$ acts on the rows. This focal manifold can be presented as $G/H$, where
$H=S(U(1)\times U(1))\times S(U(m-1)\times U(1))$, in which
the first and the third factors belong to $U(m)$ and the
other two factors belong to $U(2)$.
Denote $\pi$, $\pi_1$ and $\pi_2$ the orthogonal projection
from $\mathfrak{g}=\mathrm{Lie}(S(U(m)\times U(2)))$ to
the direc sum factors $su(m)\oplus su(2)$, $su(m)$ and $su(2)$ in $\mathfrak{g}$ respectively.

Using the method for
Case 12 or for Case 9 with $m>4$, $W$ defines a nonzero CK-vector field on the Riemannian normal homogeneous space $(SU(m)\times SU(2))/H'$ with $\mathfrak{h'}=\mathrm{Lie}(H')=\pi(\mathfrak{h})$. Direct
calculation shows that the dimension of
$\pi_1(\mathfrak{h}')=\pi(\mathfrak{h})$ ($\pi_2(\mathfrak{h}')=\pi(\mathfrak{h})$ respectively) is strictly between 0 and $\dim SU(m)$ ($\dim SU(2)$ respectively). By Proposition \ref{normal-homogeneity-CK-vector-field-semi-simple-case},
$W$ defines a nonzero CK-vector field for the Riemannian
normal homogeneous metric on either $SU(m)/H_1$ with
$\mathrm{Lie}(H_1)=\pi_1(\mathfrak{h}')$ or
$SU(2)/H_2$ with
$\mathrm{Lie}(H_2)=\pi_2(\mathfrak{h}')$.
But they are not listed in Theorem \ref{normal-homogeneity-CK-vector-field-simple-case}.

Using Theorem \ref{main-thm-1} and
Summarizing all the above cases,
we have proved the following theorem.

\begin{theorem}\label{main-thm-2}
Any homogeneous hypersurface $M$ of a unit sphere
$S^n(1)=(S^n,h)$ is isoparametric for the Randers sphere
$(S^n,F)$ of constant flag curvature, with the navigation datum $(h,W)$ in which $W$ is taken from the open $||\cdot||_h$-unit ball
in the Lie algebra of $G=I_o(S^n,M,h)$.

Furthermore, when $(S^n,F)$ is non-Riemannian, $M$ is $K$-homogeneous with $K=I_o(S^n,M,F)$ iff one of following cases happens:
\begin{description}
\item{\rm (1)} $G=SO(n)$ with an even number $n$, and $W$ is $O(n+1)$-conjugate to
    $\mathrm{diag}(0,\lambda J_{n})$ with $0<\lambda<1$.
\item{\rm (2)}
$G=SO(p)\times SO(n+1-p)$ when $p(n+1-p)$ is an even positive number, and
$W\in so(p)\oplus so(n+1-p)$ is $O(n+1)$-conjugate to
$\mathrm{diag}(\lambda_1 J_{2n_1},\lambda_2 J_{2n_2})$ with
$0<\lambda_1\leq\lambda_2<1$ or
$\mathrm{diag}(0_{n_0},\lambda J_{2n_1})$ with $0<\lambda<1$.
\item{\rm (3)}
$G$ has a one-dimensional center, i.e.
\begin{eqnarray*}
& &U(5) \mbox{ with }n=19,\\
& &SO(m)\times SO(2)\mbox{ with }m\geq 3
\mbox{ and }n=2m-1,\\ & &S(U(m)\times U(2))\mbox{ with }m\geq 2\mbox{ and }n=4m-1,\\
& &(Spin(10)\times SO(2))/\mathbb{Z}_4\mbox{ when }n=31.
\end{eqnarray*}
and $W$ is chosen from $\mathfrak{c}(\mathfrak{g})$ which must be of the form $\lambda J_{n+1}\in so(n+1)$ with $0<\lambda<1$ up to $O(n+1)$-conjugation.
\end{description}
\end{theorem}

By Lemma \ref{homogeneous-lemma}, Theorem \ref{main-thm-2} and the argument in the proof of Theorem \ref{main-thm-1}, we get the following immediate corollary.
\begin{corollary}\label{corollary-1}
If $M$ is a homogeneous hypersurface of a
non-Riemannian Randers sphere of constant flag curvature,
then it must have one, two or four distinct principal curvatures.
\end{corollary}
Here the principal curvature is the one with respect to
the localization metric $g^F_{\nabla f}$ in some neighborhood of $M$.

\section{Isoparametric hypersurface of OT-FKM type}
According to the recent progress on the classification theory for isoparametric hypersurfaces in unit spheres \cite{Ch2016} \cite{CCJ} \cite{Mi2013},
all non-homogeneous isoparametric functions (and hypersurfaces) for the unit sphere $S^n(1)=(S^n,h)$ belongs to the OT-FKM type, which can be constructed as following.

Let $\{P_0,\ldots,P_m\}$ with $m\geq 1$ be a {\it symmetric Clifford system} on $\mathbb{R}^{2l}$ (with the standard Euclidean inner product $\langle\cdot,\cdot\rangle$ and norm $|x|=\langle x,x\rangle^{1/2}$), i.e. $P_i$'s are real symmetric $2l\times 2l$-matrices
satisfying $P_iP_j+P_jP_i=2\delta_{ij}I_{2l}$ for all $i$ and $j$. Then for $x\in\mathbb{R}^{2l}$ with $|x|=1$,
$$f(x)=|x|^4-2\sum_{i=0}^m\langle P_i x,x\rangle^2$$
defines an isoparametric function on the unit sphere $S^{2l-1}(1)$, when both $m_1=m$ and $m_2=l-m-1$ are positive. We call this $f$ (or its regular level sets) an isoparametric function
(or isoparametric hypersurfaces, respectively) of the {\it OT-FKM type}. An isoparametric hypersurface of the OT-FKM type has four distinct principal curvature, with multiplicities $m_1$ and $m_2$ given above.

An isoparametric function $f$ of the OT-FKM type is determined
by the {\it Clifford sphere} $\Sigma=\Sigma(P_0,\ldots,P_m)$ defined as
$$\Sigma(P_0,\ldots,P_m)=\{P=\sum a_iP_i\mbox{ with each }a_i\in\mathbb{R}\mbox{ and }\sum a_i^2=1\}$$
rather than the particular choices for $P_i$'s.
Notice that
each $P\in\Sigma(P_0,\ldots,P_m)$ satisfy $P^2=I_{2l}$ and its eigenspaces $E_\pm(P)$ for $\pm 1$ have the same dimension $l$.

Conversely, if we assume $m=m_1\leq m_2=l-m-1$, by Theorem 4.6 in \cite{FKM1981},
the sphere $\Sigma$ can also be determined by $f$, which was
argued as following.
The corresponding focal manifold $M_-$ for the minimal value $-1$ of $f$
is an $l-1$-dimensional sphere bundle over $\Sigma$. The fiber
over each $P\in\Sigma$ is the unit sphere in $E_+(P)$. For each
$y\in M_-$, there exists a unique $P\in\Sigma$ satisfying $Py=y$. By Theorem 4.6 in \cite{FKM1981}, when $m_1\leq m_2$,
$$E_+(P)=\mathbb{R}y\oplus\mathrm{span}\{\bigcup_{N\in\perp_y M_-\backslash 0 }S_N\},$$
where $\perp_y M_-$ is the orthogonal complement of $T_yM_-$
in $T_yS^{2l-1}(1)$, and $S_N$ is the shape operator of $M_-$
at $y$, with respect to the normal vector $N$. Obviously, $E_+(P)$ and then $P$ are totally determined. So $\Sigma$ is determined by all points of $M_-$ and the
geometry of $M_-$ in $S^{2l-1}(1)$.

For each isoparametric function $f$ of the OT-FKM type, defined
by the Clifford sphere $\Sigma$. There
are many Killing vector field of the unit sphere $S^{2l-1}(1)=(S^{2l-1},h)$ which generate flows preserving $f$. They can be
used in the navigation datum $(h,W)$ for Theorem \ref{main-thm-1}. To determine all the choices for $W$, we only need
to determine the Lie algebra they span, in which the set of all possible $W$'s is the open unit $||\cdot||_h$-disk.

Firstly, the connected isometry group $SO(m+1)$ of $\Sigma$ can be lifted to a subgroup of $G$. Let $P$ and $Q$ be an orthogonal pair in $\Sigma$, i.e. $P$ and $Q$ in $\Sigma$ satisfy
$\langle P,Q\rangle=\frac{1}{2l}\mathrm{Tr}PQ=0$, or equivalently
$PQ+QP=0$. Their product $PQ$ is a skew symmetric matrix, which defined
a Killing vector field $W$ of $S^{2l-1}(1)$. The conjugations  $\mathrm{Ad}(e^{tX})$ preserve $\Sigma$, rotating the plane
spanned by $P$ and $Q$ and fixing their common orthogonal complement. So the flow generated by $W$ preserves $f$. All such $W$ spanned a Lie algebra $\mathfrak{g}'=so(m+1)$ in $\mathfrak{g}$ with the Spin action on $\mathbb{R}^{2l}$.

Secondly, there is a subalgebra $\mathfrak{c}(\Sigma)$ of $\mathfrak{g}$, spanned by all $X$'s in $so(2l)$
commuting with each $P\in\Sigma$. This subalgebra can be determined
case by case as following.

When $m$ is not a multiple of 4, up to equivalence, the
symmetric Clifford system $\{P_0,\ldots,P_m\}$ on $\mathbb{R}^{2l}$, which is a real representation for the symmetric Clifford algebra $\mathrm{Cl}^{m+1}$, can be decomposed as $k$ copies of
the unique irreducible one $\{P'_0,\ldots,P'_m\}$ on $\mathbb{R}^{2\delta_m}$, with $P_i=P'_i\otimes I_k$
for each $i$.
Denote $\mathbf{A}$ the subalgebra generated by the Clifford
system $\{P'_0,\ldots,P'_m\}$ in the matrix algebra over $\mathbb{R}^{2\delta_m}$. It is a simple algebra of the form
$\mathbb{R}(\cdot)$, $\mathbb{C}(\cdot)$ or $\mathbb{H}(\cdot)$, i.e. the matrix algebras with real, complex, quaternic coefficients.
 A matrix $X\in so(2l)$ belongs to $\mathfrak{c}(\Sigma)$ iff it commutes with all the matrices in $\mathbf{A}\otimes I_k$. We
 can use Schur's Lemma to discuss each of the following cases.

{\bf Case 1.} When $m=8q+r$ with $r=1$ or $7$, $\mathbf{A}=\mathbb{R}(2\delta_m)$ with $\delta_m=2^{4q}$ when $r=1$ and $\delta_m=2^{4q+3}$ when $r=7$.
In this case, $X\in \mathfrak{c}(\Sigma)$ iff it can be presented as $X=I_{2\delta_m}\otimes X'$ where $X'$ is skew symmetric.
So $\mathfrak{c}(\Sigma)$ is isomorphic to $so(k)$.

{\bf Case 2.} When $m=8q+r$ with $r=2$ or $6$,
$\mathbf{A}=\mathbb{C}(\delta_m)\subset\mathbb{R}(2\delta_m)$ with
$\delta_m=2^{4q+1}$ when $r=2$ and
$\delta_m=2^{4q+3}$ when $r=6$. In this case,
$X\in\mathfrak{c}(\Sigma)$ iff it can be presented as
$X=I_{\delta_m}\otimes X'$ where $X'\in \mathbb{C}(k)\subset\mathbb{R}(2k)$ is real skew symmetric, i.e. $X'\in u(k)$. So  $\mathfrak{c}(\Sigma)$ is isomorphic
to $u(k)$.

{\bf Case 3.} When $m=8q+r$ with $r=3$ or $5$,
$\mathbf{A}=\mathbb{H}(\delta_m/2)\subset\mathbb{R}(2\delta_m)$ with $\delta_m=2^{4q+2}$ when $r=3$
and $\delta_m=2^{4q+3}$ when $r=5$. In this case,
$X\in\mathfrak{c}(\Sigma)$ iff it can be presented as  $X=I_{\delta_m/2}\otimes X'$ where $X'\in\mathbb{H}(k)\subset\mathbb{R}(4k)$ is real skew symmetric, i.e. $X'\in sp(k)$. So $\mathfrak{c}(\Sigma)$ is
isomorphic to $sp(k)$.

When $m$ is a multiple of 4, there exist exactly two distinct irreducible Clifford system
$\{P'_0,\ldots,P'_{m}\}$. They are on
real vector spaces of the same dimension $2\delta_m$, which will be denoted as $\mathbf{V}_1$ and $\mathbf{V}_2$ respectively.
Denote $\mathbf{A}_i$, $i\in\{1,2\}$, the subalgebra generated by
 each irreducible Clifford system $\{P'_0,\ldots,P'_m\}$ in the matrix algebra over $\mathbf{V}_i$.

The Clifford system $\{P_0,\ldots,P_m\}$ can
be regarded as the sum of $k_1$ copies of the irreducible one
on $\mathbf{V}_1$ and $k_2$ copies of that on $\mathbf{V}_2$, where $k_1+k_2=k$ and $l=k\delta_m$. Up to congruence, we may
assume $k_1\geq k_2$.

{\bf Case 4.} When $m=8q+r$ with $r=4$, for each $\mathbf{V}_i$, $\mathbf{A}_i=\mathbb{H}(\delta_m/2)$ with $\delta_m=2^{4q+2}$. In this case, $X\in\mathfrak{c}(\Sigma)$ iff $X$ can be presented
as $X=I_{\delta_m/2}\otimes(X_1\oplus X_2)$, where for each $i$,
$X_i\in\mathbb{H}(k_i)\subset\mathbb{R}(4k_i)$ is real skew symmetric, i.e. $X_i\in sp(k_i)$. So
$\mathfrak{c}(\Sigma)$ is isomorphic to $sp(k_1)\oplus sp(k_2)$.

{\bf Case 5.} When $m=8q$, for each $\mathbf{V}_i$, $\mathbf{A}_i=\mathbb{R}(2\delta_m)$ with $\delta_m=2^{4q-1}$.
In this case, $X\in\mathfrak{c}(\Sigma)$ iff $X$ can be
presented as $X=I_{2\delta_m}\otimes (X_1\oplus X_2)$ where
$X_i\in so(k_i)$ for each $i$. So $\mathfrak{c}(\Sigma)$
is isomorphic to $so(k_1)\oplus so(k_2)$.

Finally, we determine the isomorphic type of $\mathfrak{g}=\mathrm{Lie}(I_o(S^{2l-1},M,h))$
when $m=m_1\leq m_2=l-m-1$. We have observed that
any isometry $\varphi$ in $I_o(S^{2l-1}(1),M,h)$ corresponds
to an orthogonal matrix $T\in SO(2l)$, such that
the conjugation $\mathrm{Ad}(T)$ preserves the Clifford sphere $\Sigma=\Sigma(P_0,P_1,\cdots,P_m)$.
With the metric defined by $\langle P,Q\rangle=\frac{1}{2l}\mathrm{Tr}PQ$, $\Sigma$ is identified with a unit sphere, and
$\mathrm{Ad}(T)$ defines an isometry $\tilde{\varphi}$ on $\Sigma$.
The group morphism from $\varphi$ to $\tilde{\varphi}$ is surjective, and $\varphi$ belong to the kernel iff the corresponding $T$ commutes with each $P\in\Sigma$. At the
Lie algebra level, we have the exact sequence
$$0\rightarrow\mathfrak{c}(\Sigma)
\rightarrow\mathfrak{g}\rightarrow so(m+1)\rightarrow 0,$$
and the subalgebra $\mathfrak{g}'\subset\mathfrak{g}$ (isomorphic to $so(m+1)$) defined above is a section for this exact sequence. So $\mathfrak{g}$ has the same isomorphic type as $so(m+1)\oplus\mathfrak{c}(\Sigma)$.

To summarize, we have the following theorem.
\begin{theorem} \label{main-thm-3}
Let $M$ be an isoparametric hypersurface in the unit sphere $S^{2l-1}(1)=(S^{2l-1},h)$ of the OT-FKM type, defined
by the Clifford system $\{P_0,\ldots,P_m\}$, and satisfy $m=m_1\leq m_2=l-m-1=k\delta_m-m-1$. Then $M$ is isoparametric for the Randers sphere $(S^{2l-1},F)$ when the Killing
vector field $W$ in the navigation datum $(W,h)$ is chosen from
the open $||\cdot||_h$-unit ball in
Lie algebra of $I_o(S^{2l-1},M,h)$ which is isomorphic to one of the following:
\begin{description}
\item{\rm (1)} $so(m+1)\oplus so(k)$ when $m\equiv 1$ or $7$ (mod 8),
\item{\rm (2)} $so(m+1)\oplus u(k)$ when $m\equiv 2$ or $6$ (mod 8),
\item{\rm (3)}
$so(m+1)\oplus sp(k)$ when $m\equiv 3$ or $5$ (mod 8),
\item{\rm (4)}
$so(m+1)\oplus sp(k_1)\oplus sp(k_2)$ with $k_1+k_2=k$ when $m\equiv 4$ (mod 8),
\item{\rm (5)}
$so(m+1)\oplus so(k_1)\oplus so(k_2)$ with $k_1+k_2=k$, when
$m\equiv 0$ (mod 8).
\end{description}
\end{theorem}

Almost all isoparametric hypersurfaces of the OT-FKM type
satisfy the condition $m_1=m\leq l-m-1=m_2$. Even when $m_1>m_2$, we can still find a subalgebra of $\mathfrak{g}$
with the same isomorphism type as listed in Theorem \ref{main-thm-3}, from which we can choose $W$ for the navigation and get a special isoparametric hypersurface of $(S^n,F)$.

If an isoparametric hypersurface $M\subset S^n(1)$ of the OT-FKM type satisfying $m_1>m_2$ and $(m_1,m_2)\neq (8,7)$,
either it is homogeneous, or
it is congruent to another isoparametric hypersurface of the OT-FKM
type with $m_1$ an $m_2$ switched. In either case, we can determine all the possible choices of $W$, i.e. classify the Randers metrics $F$ for the special isoparametric hypersurfaces
in Theorem \ref{main-thm-1}.

\end{document}